\documentclass{arx}
\usepackage{amsmath}
  \usepackage{paralist}
  \usepackage{graphics} 
  \usepackage{epsfig} 
 \usepackage[colorlinks=true]{hyperref}
\hypersetup{urlcolor=blue, citecolor=red}

\newtheorem{theorem}{Theorem}[section]
\newtheorem{corollary}{Corollary}

\newtheorem{lemma}[theorem]{Lemma}
\newtheorem{proposition}{Proposition}

\theoremstyle{definition}

\newtheorem{remark}{Remark}

\newcommand{\R}{{\Bbb R}}

\newcommand{\Z}{{\Bbb Z}}

\title[Separation dichotomy and  wavefronts  ]
      {Separation dichotomy and  wavefronts  for \\ a nonlinear  convolution equation}

\author[Carlos Gomez, Humberto Prado  and Sergei Trofimchuk]{}


 \email{cgomez@inst-mat.utalca.cl}
 \email{humberto.prado@usach.cl}
 \email{trofimch@inst-mat.utalca.cl}
 
\begin{document}
\maketitle

\centerline{\scshape Carlos Gomez$^1$,  Humberto Prado$^2$ \and  Sergei Trofimchuk$^1$}
\medskip
{\footnotesize
 \centerline{ 1 Instituto de M\'atematica y F\'isica, Universidad de Talca}
   \centerline{Casilla 747, Talca, Chile}
} 

\centerline{\scshape }
\medskip
{\footnotesize
 \centerline{2 Departamento de Matem\'atica, Universidad de Santiago de Chile}
   \centerline{Casilla 307, Correo-2, Santiago-Chile  }
}

\bigskip

\begin{abstract}
\noindent   This paper is concerned with a scalar nonlinear convolution equation which appears naturally  in the theory of traveling waves for monostable evolution models.  First, we prove that  each bounded positive solution of the convolution equation should either be asymptotically separated from zero or it should converge 
(exponentially) to zero.  This dichotomy principle is then used to establish a general theorem guaranteeing the uniform persistence and existence  of  semi-wavefront solutions to the convolution equation. 
Finally, we apply our  abstract results to several well-studied classes of 
evolution equations with asymmetric non-local and non-monotone response.  
We show that, contrary to the symmetric case, 
 these equations  can possess at the same time the stationary, the expansion  
and  the extinction waves.  
\end{abstract}

\section{Introduction and main results} \label{intro}

In this paper, we continue to study  the nonlinear scalar convolution equation
\begin{equation}\label{17} \hspace{-7mm}
\phi(t) = \int_Xd\mu(\tau)\int_\R K(s,\tau)g(\phi(t-s),\tau) ds, \quad t \in \R,
\end{equation}
introduced in \cite{AGT}.
Here $(X,\mu)$ is a finite measure space, an appropriate kernel $K(s,\tau) \geq 0$  is
integrable on $\R\times X$ with $ \ \int_{\R}K(s,\tau)ds>0, \ \tau \in X,$ while measurable $g: \R_+ \times X \to \R_+,$ $g(0,\tau) \equiv 0, $ is continuous in $\phi$ for every fixed $\tau \in X$ and there exists $g'(0,\tau)>0$. Our   goal here   is to establish a satisfactory criterion for  the existence  of  semi-wavefronts (i.e. positive, bounded,  and vanishing  at either $+\infty$ or $-\infty$ solutions) to (\ref{17}). Then in Section \ref{apn} we will apply this criterion  to two non-local and asymmetric  monostable evolution equations.  In this way, we  develop  further some ideas from  \cite{TAT}.  It should be noted that equation (\ref{17}) is one of  valid general forms for the description of  traveling wave profiles.   Other similar yet non-equivalent  functional equations can be found in \cite{AR,dkMB,dk,KS,VL,WBR}. 

It was shown in \cite{AGT} that  the characteristic function
$$
\chi(z):= 1-\int_X\int_\R K(s,\tau)g'(0,\tau)d\mu(\tau) e^{-sz}ds.
$$
plays a key role in the investigation of equation (\ref{17}).  In particular, the following holds (see \cite[Theorem 2]{AGT}):
\begin{proposition} \label{gg} Assume $\chi(0)<0$.  Let $\phi:\R\rightarrow[0,+\infty)$ be a bounded solution  to  equation (\ref{17}). If $\phi(-\infty)=0$  and $\phi(t)\not\equiv 0, \ t \leq t'$ for each fixed $t'$, then $\chi(z)$  is well defined and has a zero on some non-degenerate interval $(0, \gamma]$.  
\end{proposition}
And as we will prove below under  the additional mild conditions 
\begin{description}
\item[{\rm \bf(C)}]  For each $\delta >0$ there is  a measurable  $C_\delta(\tau)\geq 0$ such that 
$$g(u,\tau)\leq C_\delta(\tau)u,  \ u\in[0,\delta], \quad   \int_X C_\delta(\tau)d\mu(\tau)\int_\R K(s,\tau)ds<+\infty; $$ 
\item[{\rm \bf(P)}] Bounded solution $\phi(t)\geq 0$  of (\ref{17})  vanishes  at some point only if   $\phi(t) \equiv 0$, 
\end{description}
the conclusion of Proposition  \ref{gg} remains true even if we replace assumption 
$\phi(-\infty)=0$ by a weaker $\liminf_{t\to -\infty}\phi(t)=0$.  Moreover,  in Theorem \ref{thm:per} below  we  prove the equivalence of these two  properties for solutions of equation (\ref{17}).  In view of Theorem 1 and Lemma 3 from \cite{AGT},  this result has the following nice consequence:  under a few natural restrictions on $K, g$,  each bounded positive solution $\phi$ with $\liminf_{t\to -\infty}\phi(t)=0$ converges  exponentially  to zero at $-\infty$. 

Note that assumption {\rm \bf(P)}  can be easily checked due to 
\begin{lemma}\label{lem:nucleo}
Assume that there are  $\widetilde{X}\subset X$, $\mu(\widetilde{X})>0,$ and a measurable  $A: \widetilde{X} \to (0, +\infty)$  such that  $\tau\in\widetilde{X}$ implies   (i) $g(u,\tau)=0$ if and only if $u=0$; (ii) $K(s,\tau)>0$ for all $s\in(-A(\tau),A(\tau))=: I_\tau$. Then $\phi(0)=0$ implies $\phi(t)\equiv 0$. 
\end{lemma}
\begin{proof}  Suppose that $0=\phi(0)=\int_X d\mu(\tau)\int_\R K(s,\tau)g(\phi(-s),\tau)ds$. Then we have $K(s,\tau)g(\phi(-s),\tau)=0$ almost everywhere on $\widetilde X \times \R$. Hence, for some $\tau_0\in\widetilde{X}$,  we obtain that $g(\phi(-s),\tau_0)=0$ for all  $s\in I_{\tau_0}$. Thus $\phi(-s)=0$, $s\in I_{\tau_0}$. Similarly, if $\phi(t_0)=0$ for some $t_0\in\R$, then $\phi(t)=0$ for all $t$ in an open neighborhood of $t_0$. In consequence, the set of zeros of continuous $\phi$ is open and closed, and we may conclude that $\phi\equiv 0$. 
\end{proof}
We are ready to state our first main result:  
\begin{theorem}\label{thm:per} Assume {\rm \bf(C)}, {\rm \bf(P)} with $\chi(0) <0$. Then the following dichotomy holds for each bounded solution $\phi(t)\geq 0$  of (\ref{17}):  either  $\liminf_{t\to +\infty}\phi(t)>0$ or $\phi(+\infty)=0$.  The similar alternative is also valid at $-\infty$.   
\end{theorem}
An easy  combination of results from Proposition \ref{gg} and Theorem \ref{thm:per} leads to 
\begin{corollary}\label{pu}
If  $\chi(z)$ does not have any positive [negative] zero and $\phi$ is a positive bounded solution of  (\ref{17}), then $\liminf\limits_{t\to -\infty}\phi(t)>0$ [respectively, $\liminf\limits_{t\to +\infty}\phi(t)>0$].  As a consequence,  equation  (\ref{17}) can not have positive pulse solutions (i.e. solutions satisfying $\phi(-\infty) = \phi(+\infty)=0$). 
\end{corollary}
\begin{proof} 
Since $\chi(0) <0$ and
$\chi$ is concave on its maximal domain of  definition,  all real zeros of $\chi $ should be of the same sign (if they exist). \end{proof}
Let $\omega$ denote either $+\infty$ or $-\infty$. By Corollary \ref{pu}, we have the following point-wise persistence property:  for each  bounded positive solution $\phi(t)$ of Eq. (\ref{17}) satisfying $\phi(-\omega)=0$  there is  some $\delta(\phi) >0$ such that $\liminf_{t\to \omega}\phi(t)\geq \delta(\phi)$.   This fact allowed us  to exclude the latter inequality from the definition of semi-wavefronts (cf. with boundary conditions (1.6) in \cite{BNPR}).  Now,  in order to prove the uniform persistence (this means that the above mentioned $\delta(\phi)$ can be chosen independent of $\phi$) as well as the existence of solutions to equation (\ref{17}), we will impose additional conditions  on its nonlinearity: 
\begin{description}
\item[{\rm \bf(N)}]  N1. There exists $\tau_0\in X, \mu(\tau_0)=1,$ such that 
$g(v,\tau)$ increases in $v$ for 

each  fixed $\tau\not=\tau_0$ and $g(v,\tau_0) >0, \ v >0$. 
Consider the monotone function $$\tilde{g}(v):=\int_{X\setminus\{\tau_0\}}g(v,\tau)d\mu(\tau)\int_{\R}K(s,\tau)ds.$$ 

N2.   There exists  $\zeta_2>0$ such that  $\Theta(v):=v-\tilde{g}(v)$ is strictly  increasing on $[0,\zeta_2]$, and  
$\Theta(\zeta_2)>C\max_{v\geq 0}g(v,\tau_0)$  
where $C:= \int_{\R}K(s,\tau_0)ds$.
\end{description}
Set $G(v): = \Theta^{-1}(Cg(v,\tau_0)).$  It is clear that $G(0)=0, \ 0 < G(v) < \zeta_2, \ v >0,$ and that the graphs of $G(v)$ and $g(v,\tau_0)$ have similar geometrical shapes. In particular, they share the same critical points. 

\begin{figure}[h]\label{F1}
\begin{center}
\includegraphics[scale=0.7]{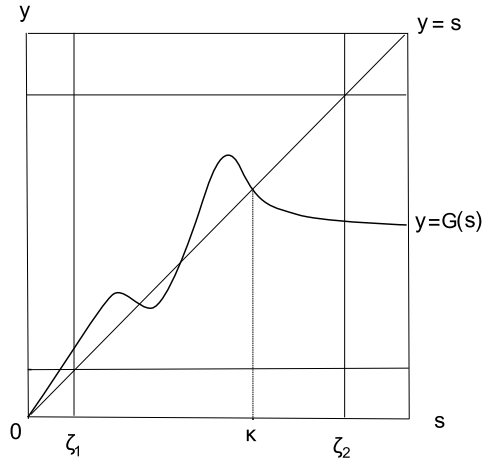}
\end{center}
\caption{Nonlinearity $G$ under hypotheses {\rm \bf(N)} and $\chi(0)<0$.}
\end{figure}

If $\varphi(t)=c$ is a constant solution of (\ref{17}), then $c={G}(c)$ because of the relation 
\begin{equation*}
c=\tilde{g}(c)+g(c,\tau_0)\int_\R K(s,\tau_0)ds=\tilde{g}(c)+C g(c,\tau_0) = c- \Theta(c)+C g(c,\tau_0). 
\end{equation*}

Several  additional important properties of $G$ are listed below: 
\begin{lemma} \label{ran} Let $\chi(0)< 0$ and  {\rm \bf (C), (N)} hold. Then, for some $\zeta_1 \in (0, \zeta_2)$,  \\
1. $G \in C(\R_+,\R_+)$ is positive for $s
> 0$ and there exists $G'(0+)>1$;\\
2. $G([\zeta_1,\zeta_2])\subseteq [\zeta_1,\zeta_2]$ and
$G(\R_+)\subseteq [0,\zeta_2]$;\\
3. $\min_{s \in [\zeta_1,\zeta_2]}G(s) = G(\zeta_1)$ while $G(s) > s$ for $s \in (0, \zeta_1]$. 
\end{lemma}
\begin{proof} Let us  show, for instance,  that $G'(0+)>1$.  In view of  {\rm \bf (C)}, this derivative exists and  
is equal to $Cg'(0,\tau_0)/(1- \tilde g'(0))$. Thus  $G'(0) >1$ if and only if $\chi(0) <0$.    Observe that 
$\tilde g'(0+) \leq 1$ since $\Theta'(0+)\geq 0$ and we do not exclude the case $G'(0+)=+\infty$. 
\end{proof}
Using the above framework, we can improve  conclusions  of Theorem \ref{thm:per}: 
\begin{theorem}\label{thm:per_unif}
Assume $\mathbf{(N)}$ along with all conditions  of Theorem \ref{thm:per} and take $\zeta_1>0$ as in Lemma \ref{ran}. Let $\phi$ be  a positive bounded solution  of equation (\ref{17}). If $m=\inf_{s\in\R}\phi(s)<\zeta_1$ then $\lim_{t\to \omega}\phi(t)=0$ and $\liminf_{t\to -\omega}\phi(t)> \zeta_1$  for some $\omega \in \{-\infty, +\infty\}$.
\end{theorem}
Our third result can be considered as a further development of  Theorem 6.1 from \cite{dkMB} which was proved for a single-point space $X$ and under  more restrictive conditions on the nonlinearity $g$:
\begin{theorem}\label{thm:existencia} Assume {\rm \bf(N)}, that  $G'(0)$ is finite and that  $g(s,\tau)\leq g'(0,\tau)s$ for all $s\geq 0,$ $\tau \in X.$ If  $\chi(z), \chi(0) < 0, $ is defined and  changes its sign on some open interval $(0, \bar \omega)$ [respectively, on $(-\bar \omega,0)$],  then equation (\ref{17}) has at least one semi-wavefront, with $\varphi(-\infty)=0, \ \sup_{s \in \R} \varphi (s) \leq \zeta_2$  and $\liminf_{t\rightarrow +\infty}\varphi(t)> \zeta_1$ [respectively, with   $\varphi(+\infty)=0, \ \liminf_{t\to -\infty}\varphi(t)>\zeta_1$].  
Moreover, if 
equation $G(s)=s$ has exactly two solutions $0$
and $\kappa$ on $\R_+$, and the point $\kappa$ is globally attracting with respect to the map $G: (0,\zeta_2] \to (0,\zeta_2] $ then $\phi(+\infty) = \kappa$. 
\end{theorem}
\vspace{0mm}
\begin{remark} It is worth noting that the existence of $g'(0,\tau)$  (and consequently of $G'(0)$) is not at all obligatory for the existence of semi-wavefronts.  Indeed, suppose that  there is a measurable $l(\tau)$ satisfying
$g(s,\tau) \leq l(\tau)s, s \geq 0,$ and consider associated characteristics 
$$
\chi_l(z):= 1-\int_X\int_\R K(s,\tau)l(\tau)d\mu(\tau) e^{-sz}ds, \ \tilde g_l':=  \int_{X\setminus\{\tau_0\}}\int_\R K(s,\tau)l(\tau)d\mu(\tau)ds. 
$$
We assume also that  {\rm \bf(N)} holds, $G$ possesses the second and the third properties of Lemma  \ref{ran},  and $\tilde g_l' <1$ (this generalizes assumption $G'(0)\in \R$). Then all conclusions of Theorem \ref{thm:existencia} remain valid if we replace in its formulation $\chi$ with $\chi_l$.  See the second part of Section \ref{SE} for more details. 
\end{remark}
The paper is organized as follows. In Section \ref{s2p},  we
prove the dichotomy principle. The first part of Section \ref{UPP} shows 
how to avoid possible troubles with unbounded solutions of the convolution equation.  The second part of the same section presents a short proof of the uniform persistence property.  These preliminary results are essential for proving the  existence theorem in Section \ref{SE}. Finally,  several applications are considered in the last section of the paper.  
Associated characteristic equation is analyzed in Appendix. 

\section{The proof of the dichotomy principle (Theorem \ref{thm:per})} \label{s2p} \hfill . \vspace{2mm}

1. Let $\phi(t)$ be a bounded  solution of (\ref{17}).  It is easy to see that $\phi(t)$ is uniformly continuous on $\R$. Indeed,  setting $\delta = |\phi|_\infty,$ we find that 
\begin{eqnarray*}
&&|\phi(t+h)-\phi(t)|\leq\int_X d\mu(\tau)\int_\R|K(s+h,\tau)-K(s,\tau)|g(\phi(t-s),\tau)ds\\
&&\leq |\phi|_\infty\int_X C_{\delta}(\tau)d\mu(\tau)\int_\R|K(s+h,\tau)-K(s,\tau)|ds =: |\phi|_\infty \sigma_\delta(h),
\end{eqnarray*}
where  $\lim_{h \to 0} \sigma_\delta(h) = 0$  because of the continuity of translation in $L_1(\R)$ and the Lebesgue's 
dominated convergence theorem.  

2. Next we prove an analog of Proposition \ref{gg} when  $\phi(+\infty) =0$ and $\phi$ is bounded and positive. We have
$$\phi(-t)=\int_Xd\mu(\tau)\int_{\R}K(s,\tau)g(\phi(-t-s),\tau)ds,\quad t\in\R.$$
Set $\psi(t):=\phi(-t)$, then $\psi(-\infty)=0$ and
\begin{equation}\label{eq:two}
\psi(t)=\int_X d\mu(\tau)\int_\R K(-s,\tau)g(\psi(t-s),\tau)ds.
\end{equation}
Let $\chi(z)$ [$\chi_1(z)$] be characteristic equation for Eq. (\ref{17}) [Eq. (\ref{eq:two}), respectively]. 
We have
\begin{eqnarray*}
\chi_1(z)&=&1-\int_X\int_\R K(-s,\tau)g'(0,\tau)d\mu(\tau) e^{-sz}ds\\
&=&1-\int_\R \int_X K(s,\tau)g'(0,\tau)d\mu(\tau)e^{sz}ds=\chi(-z)
\end{eqnarray*}
and  thus $\chi_1(0)=\chi(0) < 0$.  By Proposition \ref{gg},   $\chi_1(z)$  has at least one positive root. 
Therefore $\chi(z)$ has at least one negative zero. 

3. Now, let suppose that  $\limsup_{t\rightarrow+\infty}\phi(t)=S>0$ and  $\liminf_{t\rightarrow+\infty}\phi(t)=0$. Since $\chi(0) <0$ and
$\chi$ is concave on its maximal domain of definition,  all real zeros of $\chi $ should be of the same sign (if they exist).   Suppose that $\chi$ does not have any real  negative [{\sf respectively, positive}] root. 
For a fixed $j>S^{-1}$ there exists a sequence of intervals $[p_i,q_i]$, $\lim p_i=+\infty$, such that $\phi(p_i)=1/j$, $\lim \phi(q_i)=0$  [{\sf respectively, $\phi(q_i)=1/j$, $\lim \phi(p_i)=0$}]  and $\phi(t)\leq 1/j$, $t\in[p_i,q_i]$ . Note that $\limsup_{i\rightarrow+\infty}(q_i-p_i)=+\infty$. Indeed, otherwise we can suppose that $\lim_{i\rightarrow+\infty}(q_i-p_i)=\sigma>0$. By the pre-compactness of $\{\phi(t+s);s\in\R\}$ in the compact-open topology of $C(\R)$, the sequence $w_i(t):=\phi(t+p_i)$ [{\sf respectively, $w_i(t):=\phi(t+q_i)$}]  of solutions to Eq. (\ref{17}) contains a subsequence converging to a non-negative bounded function $w_*(t)$ such that $w_*(0)=1/j$, $w_*(\sigma)w_*(-\sigma)=0$.  Since, due to the Lebesgue's   dominated convergence theorem, $w_*(t)$ satisfies (\ref{17}) as well, this contradicts to {\rm \bf(P)}. Thus $q_i-p_i\rightarrow+\infty$ and we can suppose that $w_i(t)$ has a subsequence converging to a bounded positive solution $w_*(t)$ of (\ref{17}) satisfying $0<w_*(t)\leq 1/j$ for all $t\geq 0$  [{\sf respectively, for all $t \leq 0$}]. Since $w_*(+\infty)=0$  [{\sf respectively, $w_*(-\infty)=0$}]   is impossible due to Proposition \ref{gg} and the second step of the proof, we conclude that $0<S^*=\limsup_{t\rightarrow +\infty}w_*(t)\leq 1/j$ [{\sf respectively, $0<S^*=\limsup_{t\rightarrow -\infty}w_*(t)\leq 1/j$}]. Let $r_i\rightarrow+\infty $ [{\sf respectively, $r_i \to -\infty$}] be such that $w_*(r_i)\rightarrow S^*$, then $w_*(t+r_i)$ has a subsequence converging to a positive solution $\zeta_j:\R\rightarrow[0,1/j]$ of (\ref{17}) such that $\max_{t\in\R}\zeta_j(t)=\zeta_j(0)=S^*\leq 1/j$. Now, let us consider $y_j(t)=\zeta_j(t)/\zeta_j(0)$.  Each $y_j$ satisfies
\begin{equation}\label{eq:linealj}
y_j(t)=\int_Xd\mu(\tau)\int_\R K(s,\tau)a_j(t-s, \tau)y_j(t-s)ds,
\end{equation}
where $a_j(t, \tau)=g(\zeta_j(t),\tau)/\zeta_j(t)$.
We claim that $\{y_j(t)\}$ has a subsequence converging to a continuous solution $y_*:\R\rightarrow [0,1], \ y_*(0)=1,$ of equation  
\begin{equation}\label{eq:lineal}
y_*(t)=\int_X g'(0,\tau)d\mu(\tau)\int_\R K(s,\tau)y_*(t-s)ds.
\end{equation}
Indeed,  the sequence $\{y_j(t)\}_{j=1}^{+\infty}$ is equicontinuous  because of  
\begin{eqnarray*}
\hspace{-6mm} |y_j(t+h)-y_j(t)|\leq\int_X d\mu(\tau)\int_\R a_j(t-s)y_j(t-s)|K(s+h,\tau)-K(s,\tau)|ds\\
\hspace{-6mm} \leq\int_Xd\mu(\tau)\int_\R a_j(t-s)|K(s+h,\tau)-K(s,\tau)|ds \leq \sigma_1(h), \label{cota}
\end{eqnarray*}
where $\sigma_\delta$ was defined on step 1. In addition, 
$$
\left|\int_\R K(s,\tau)a_j(t-s, \tau)y_j(t-s)ds\right|\leq C_1(\tau)\int_\R K(s,\tau)ds \in L_1(X),
$$ 
so that,  by the Lebesgue's dominated convergence theorem, we can pass to the limit (as $j\rightarrow\infty$) in (\ref{eq:linealj}).  Hence, our claim is proved. 

4. The proof of Theorem  \ref{thm:per}  will be finalized, if we show that  (\ref{eq:lineal}) cannot have any nontrivial  continuous solution 
$y_*\geq 0$.  Since $$\int_X g'(0,\tau)d\mu(\tau)\int_\R K(s,\tau)ds>1$$ there exists $N>0$ such that $$\rho:=\int_X g'(0,\tau)d\mu(\tau)\int_{-N}^N K(s,\tau)ds>1.$$  
Integrating equation  (\ref{eq:lineal}) between $t'$ and $t> t'$, we obtain 
\begin{eqnarray*}\hspace{-9mm}
\int_{t'}^{t}y_*(v)dv\geq &&
 \int_Xg'(0,\tau) d\mu(\tau)\int_{-N}^NK(s,\tau)\int^t_{t'}y_*(v-s)dvds
 \\
=&& \int_Xg'(0,\tau)d\mu(\tau)\int_{-N}^NK(s,\tau)(\int^{t'}_{t'-s}+ \int^{t}_{t'}+\int^{t-s}_{t})y_*(v)dvds,\end{eqnarray*}
from which 
$$
\int_{t'}^{t}y_*(v)dv\leq \frac{2 \int_X\int_{-N}^N |s|K(s,\tau)g'(0,\tau)dsd\mu(\tau)}{\int_X\int_{-N}^N K(s,\tau)g'(0,\tau)dsd\mu(\tau)-1}, \quad t' < t. 
$$
Therefore $y_* \in L_1(\R)$.  Now we easily get a contradiction by integrating (\ref{eq:lineal}) over the real line: 
$$
0 < \int_\R y_*(v)dv=\left[\int_X g'(0,\tau)d\mu(\tau)\int_\R K(s,\tau)ds\right] \int_\R y_*(v)dv.  
$$
Hence, the dichotomy principle of  Theorem \ref{thm:per} is established at $ +\infty$. The other case  can be reduced to the previous one by doing the change of variables $\psi (t) := \phi(-t)$ and considering equation 
(\ref{eq:two}) with $\chi_1$ instead of  (\ref{17}) with $\chi$.  \hfill $\square$

\section{The uniform permanence property} \label{UPP}
\subsection{The uniform boundedness of solutions.} \label{s31}
It should be noted that, in general, equation (\ref{17}) might have unbounded continuous solutions.  Corresponding examples can be constructed by taking appropriate linear $g(u,\tau)$.  
Nevertheless, as we show in the continuation,   with conditions  ({\bf N})  and $\chi(0)<0$ being  assumed, 
it is easy to avoid eventual troubles with unbounded solutions  in the following two ways: \\
\underline {\bf Modification of  the convolution equation.}  
Consider $$\bar{g}(u,\tau)=\min\{g(u,\tau),g(\zeta_2,\tau)\},\,\tau\neq\tau_0,\,\,\bar{g}(u,\tau_0):=g(u,\tau_0)$$ and $$\bar{g}(v)=\int_{X\setminus\{\tau_0\}}\bar{g}(v,\tau)d\mu(\tau)\int_R K(s,\tau)ds.$$ Then $\bar\Theta(s):=s-\bar{g}(s)$ is a strictly increasing function. Indeed, $\bar{\Theta}(s)=\Theta(s)$, $0\leq s\leq \zeta_2,$ and we know that $\Theta(s)$ strictly increases in $[0,\zeta_2]$.  Furthermore, for $s\geq \zeta_2$, we have
$\bar{\Theta}(s)=s-\bar{g}(s)=s-\bar{g}(\zeta_2)$ where $\bar{g}(\zeta_2)$ is a constant. Hence $\bar{\Theta}(s)$ is strictly increasing on $\R_+$.
If we set $\bar{G}(v)=\bar{\Theta}^{-1}(C\bar{g}(v,\tau_0))$, we find that $\bar{G}(v)= G(v) \leq \zeta_2$ 
for $v\geq 0$.

Let us consider now a modified convolution equation
\begin{equation*}
\phi(t)=\int_X d\mu(\tau)\int_\R K(s,\tau)\bar{g}(\phi(t-s),\tau)ds.
\end{equation*}
Each its solution $\phi(t)$ is bounded;
\begin{equation*}
\phi(t)\leq\bar{g}(\zeta_2)+C\max_{v\geq 0}g(v,\tau_0)<\bar{g}(\zeta_2)+\Theta(\zeta_2)=\zeta_2.
\end{equation*}
The latter estimate assures that $\phi(t)$  simultaneously satisfies   (\ref{17}).

\underline {\bf  Subexponential solutions.}  Assume additionally that 
\begin{equation}\label{sl}
g(u,\tau)\leq g'(0,\tau)u, \ u \geq 0,  \ {\rm for \ each \ } \tau\neq\tau_0.
\end{equation}
If, for some $\lambda >0$, $\phi$ satisfies (\ref{17}) and $\phi(t) \leq \delta e^{\lambda t}, t \in \R$, then
\begin{equation}\label{eq:lineal2}
\phi(t)\leq\int_{X\setminus\{\tau_0\}}d\mu(\tau)\int_\R K(s,\tau)g'(0,\tau)\phi(t-s)ds+\rho
\end{equation}
where $\rho:=\sup_{u\geq 0}g(u,\tau_0)\int_\R K(s,\tau_0)ds \leq \Theta(\zeta_2)$. Suppose, in addition, that $$\theta:=\int_{X\setminus\{\tau_0\}}d\mu(\tau)\int_\R K(s,\tau)g'(0,\tau)e^{-\lambda s}ds<1$$ and $\gamma:=\int_{X\setminus\{\tau_0\}}d\mu(\tau)\int_\R K(s,\tau)g'(0,\tau)ds<1$.
The first inequality holds automatically if $\chi(\lambda)=0$ because of  $\int_\R K(s,\tau_0)g'(0,\tau_0)e^{-\lambda s}ds >0$.  Similarly, since $\gamma = \tilde g'(0)$,  the second inequality holds whenever  $G'(0+)$ is finite. 
\begin{lemma} \label{bs2}
If (\ref{sl}) holds, $\chi(\lambda) =0$ and $G'(0)$ is a finite number  then each solution $\phi(t) \leq \delta e^{\lambda t}$ of (\ref{17}) is bounded. In fact, $$0\leq\phi(t)\leq \min\{\zeta_2, \sup_{u\geq 0}g(u,\tau_0) \frac{G'(0)}{g'(0,\tau_0)}\},\quad  t\in\R.$$
\end{lemma}
\begin{proof}
Using $\phi(t)\leq\delta e^{\lambda t}$  in (\ref{eq:lineal2}) and arguing by induction, we find that $$\phi(t)\leq\delta e^{\lambda t}\theta^n+\rho+ \rho\gamma+\rho\gamma^2+\ldots+\rho\gamma^n.$$ Then, by passing to the limit as $n\rightarrow\infty$, we obtain the required estimate.  We recall here that $\gamma = \tilde g'(0)$, $G'(0)= Cg'(0,\tau_0)/(1- \tilde g'(0))$ and 
$C=  \int_{\R}K(s,\tau_0)ds$.  The inequality $\phi(t) \leq \zeta_2$ follows from Lemma \ref{vg} proved in continuation.  
\end{proof}
\subsection{The proof of the uniform persistence (Theorem \ref{thm:per_unif})}  
 Let $\phi$ a bounded positive solution of the equation (\ref{17}).  Set $$0\leq m:=\inf_{t\in\R}\phi(t)\leq\sup_{t\in\R}\phi(t)=: M<+\infty.$$
\begin{lemma}\label{vg} $[m,M] \subseteq G([m,M])$.
\end{lemma}
\begin{proof} Let $\{t_j\}$ be such that $M_j:= \phi(t_j) \to M$. 
We have
\begin{eqnarray*}
&&\phi(t_j)=M_j\leq\int_X\max_{v\in[m,M]}{g(v,\tau)}d\mu(\tau)\int_{\R}K(s,\tau)ds\\
&&=\max_{v\in[m,M]}\int_{X\setminus\{\tau_0\}}g(v,\tau)d\mu(\tau)\int_{\R}K(s,\tau)ds+\max_{v\in[m,M]}g(v,\tau_0)\int_{\R}K(s,\tau_0)ds\\
&&=\tilde{g}(M)+\max_{v\in[m,M]}g(v,\tau_0)\int_{\R}K(s,\tau_0)ds.
\end{eqnarray*}
Thus  $M\leq \max_{v\in[m,M]}G(v)$. Similarly,  $m\geq \min_{v\in[m,M]}G(v)$. 
\end{proof}
Now, assumption $\mathbf{(N)}$, $G'(0)>1$  and ${m} < \zeta_1$ yield  ${m} = 0$, cf. Fig. \ref{F1}.  
Hence, due to the positivity of $\phi(t)$, there exists 
$\omega \in \{-\infty, +\infty\}$ such that 
$\liminf_{t \to \omega} \phi(t)=0$. Then, applying Theorem \ref{thm:per} and Corollary \ref{pu}, we find that  $\phi(\omega) = 0$ and $\mu:= \liminf_{t \to -\omega} \phi(t) >0$.  Making use of our standard limiting solution argument, we see that, 
for some $t_j \to -\omega$, the sequence $\phi(t+t_j)$ is converging in the compact-open topology of $C(\R)$ to
some function $\phi_1(t),\   \mu: = \inf_{t \in \R} \phi_1(t) \leq \sup_{t\in \R} \phi_1(t) \leq M$  solving 
equation (\ref{17}).  By Lemma \ref{vg}, we have $[\mu,{M}] \subseteq G([\mu,{M}])$ which implies 
$\mu > \zeta_1$.    \hfill $\square$
\begin{remark} \label{rs}The last argument in the proof of Lemma \ref{vg} shows also  that $[m',M'] \subseteq G([m',M'])$, where 
 $m':=\liminf_{t\to\omega}\phi(t)\leq\lim\sup_{t\to\omega}\phi(t)=: M'$ and $\omega \in \{-\infty, +\infty\}$. 
\end{remark}
\section{The proof of the existence} \label{SE}
Throughout all this section, we are assuming that $\mathbf{(N)}$ holds,  $\chi(0)<0$ and   
\begin{equation}\label{zab}
g(s,\tau)\leq g'(0,\tau)s \ {\rm \  for\ all\ } \  s\geq 0, \  \tau \in X.
\end{equation}
1. For a moment, let us suppose additionally  that 
\vspace{2mm}

${\rm \bf(L)}$ $g:(0,\infty)\times X\rightarrow(0,+\infty)$ is bounded and  uniformly linear in some right neighborhood of the origin: $g(s,\tau)=g'(0,\tau)s$, $s\in[0,\delta)$,  $\tau \in X$.

\vspace{2mm}

Let $\lambda \in (0, \bar \omega)$ be the leftmost  positive solution of equation $\chi(z)=0$, and set 
\begin{eqnarray*}
X&:=&\{\varphi\in C(\R,\R):\|\varphi\|=\sup_{s\leq 0}e^{-0.5\lambda s}|\varphi(s)|+\sup_{s\geq 0}e^{-\nu s}|\varphi(s)|<+\infty\};\\
\mathfrak{ K }&:=&\{\varphi\in X; \phi^-(t)=\delta e^{\lambda t}(1-e^{\epsilon t})\chi_{\R_{-}}(t)\leq\varphi(t)\leq\delta e^{\lambda t}=\phi^+(t), t\in\R\}, 
\end{eqnarray*}
where $\epsilon >0$ and $\nu:=\lambda+\epsilon< \bar \omega$ are such that $\chi(\nu) >0$. We want to prove the existence of fixed points $\varphi$, $\varphi\in\mathfrak{K}$, $\sup_{s \in \R}\varphi(s)<+\infty$,  to the operator
$$\mathcal{A}\varphi(t)=\int_Xd\mu(\tau)\int_\R K(s,\tau)g(\varphi(t-s),\tau)ds.$$
A formal linearization of $\mathcal{A}$ along the trivial steady state is given by
$$L\varphi(t)=\int_X d\mu(\tau)\int_{\R}K(s,\tau)g'(0,\tau)\varphi(t-s)ds.$$  
\begin{eqnarray*} \hspace{-9mm}  \mbox{We have that }
L\phi^+(t)&=&\int_X d\mu(\tau)\int_\R K(s,\tau)g'(0,\tau)\delta e^{\lambda(t-s)}ds\\
&=&\delta e^{\lambda t}\int_X g'(0,\tau)d\mu(\tau)\int_\R K(s,\tau)e^{-\lambda s}ds=\delta e^{\lambda t}=\phi^+(t).
\end{eqnarray*}
On the other hand, $L\phi^-(t)>\phi^-(t)$, $t\in\R$.  Indeed, we have, for a fixed  $t\leq 0$, 
\begin{eqnarray*}\hspace{-0.7cm}
\delta^{-1}L\phi^-(t)&=&\int_Xd\mu(\tau)\int_t^{+\infty}K(s,\tau)g'(0,\tau)(e^{\lambda(t-s)}-e^{\nu(t-s)})ds\\
&\geq&\int_Xd\mu(\tau)\int_\R K(s,\tau) g'(0,\tau)(e^{\lambda(t-s)}-e^{\nu(t-s)})ds\\
&=&e^{\lambda t}-e^{\nu t}(1-\chi(\nu))=e^{\lambda t}-e^{\nu t}+e^{\nu t}\chi(\nu)> e^{\lambda t}-e^{\nu t}=\delta^{-1}\phi^-(t).
\end{eqnarray*}
\begin{lemma} \label{pcc}
$\mathfrak{K}$ is a closed, bounded, convex subset of $X$ and $\mathcal{A}:\mathfrak{K}\rightarrow \mathfrak{K}$ is a completely continuous map.
\end{lemma}
\begin{proof} It is clear that $\mathfrak{K}$ is a closed, bounded, convex subset of $X$. To prove that $\mathcal{A}(\mathfrak{K}) \subseteq \mathfrak{K}$,   we observe first  that, for $\varphi \in \mathfrak{K}$, 
\begin{eqnarray*}
\mathcal{A}\varphi(t)&\leq&\int_Xd\mu(\tau)\int_\R K(s,\tau)g'(0,\tau)\varphi(t-s)ds=L\varphi(t)\leq L\phi^+(t)= \phi^+(t).
\end{eqnarray*}
Next,  if for some $u$ we have that $0<\phi^-(u)\leq\varphi(u)$, then $u<0$ so that $\varphi(u)\leq\delta e^{\lambda u}\leq\delta$, which implies that $g(\varphi(u),\tau)=g'(0,\tau)\varphi(u)$. If $\phi^-(u)=0$ then $g(\varphi(u),\tau)\geq g'(0,\tau)\phi^-(u)=0$. In either case, 
\begin{eqnarray*}
\mathcal{A}\varphi(t)&\geq&\int_X d\mu(\tau)\int_\R K(s,\tau)g'(0,\tau)\phi^-(t-s)ds = L\phi^-(t)
>\phi^-(t).
\end{eqnarray*}
Now, we claim that $\mathcal{A}\mathfrak{K}$ is a precompact subset of $\mathfrak{K}$. Indeed, the convergence 
in $\mathfrak{K}$ is the uniform convergence on compact subsets of $\R$. On the other hand, the set of functions from $\mathcal{A}\mathfrak{K}$  restricted on every fixed compact interval $[-k,k]$ is obviously uniformly bounded and is also equicontinuous in virtue of the estimation (uniform with respect to $ t \in [-k,k], \phi \in \mathfrak{K}$):
$
|\mathcal{A}\phi(t+h)-\mathcal{A}\phi(t)| \leq$ 
\begin{equation}\label{esc}
\delta e^{\lambda k} \int_X g'(0,\tau)d\mu(\tau)\int_\R|K(s+h,\tau)-K(s,\tau)| 
e^{-\lambda s}ds
\to 0, \ h \to 0.
\end{equation}
Finally, the continuity of $\mathcal{A}$ in  $\mathfrak{K}$  can be easily established by using the dominated convergence theorem and the compactness property of $\mathcal{A}$.
\end{proof} 
Then Lemmas \ref{bs2}, \ref{vg}, \ref{pcc} and the Schauder's fixed point theorem yields 
\begin{theorem}\label{thex} Assume ${\rm \bf(L)}$ and let $\lambda$ be the leftmost  positive zero of  $\chi$.
Then $\mathcal{A}$ has at least one fixed point $\phi$ in $\mathfrak{K}$. If $G'(0)$ is a finite number then  
$|\phi|_{\infty}:= \sup_{s \in R}\phi(s)$  is also finite and  $|\phi|_{\infty} < \zeta_2$.  Moreover, if  the point $\kappa$ is globally attracting with respect to the map $G: (0,\zeta_2] \to (0,\zeta_2] $ then $\phi(+\infty) = \kappa$. 
\end{theorem}
It should be noted that the last statement of this theorem is a straightforward  consequence of Remark  \ref{rs} 
(see also \cite{GBT} where various conditions assuring  the global stability property of $G$ are given).

2.   Next  we show how to reduce the general situation  to the case studied in the first part of this section. Consider the sequence of measurable functions 
$$\gamma_n(s,\tau):=\left\{\begin{array}{ll}g'(0,\tau)s,& \textrm{ for } s\in[0,1/n],\\
 \max\{g'(0,\tau)/n, g(s,\tau)\},& \textrm{ when } s \geq 1/n, \end{array} \right.$$
all of them continuous in $s$ for each fixed $\tau$ and satisfying hypothesis ${\rm \bf(L)}$ with $\delta = 1/n$.  
Note that $\gamma_n(s,\tau)$ converges uniformly to $g(s,\tau)$ on $\R_+$ for every fixed $\tau$.  
Next, set $X':=X\setminus\{\tau_0\}$ and consider continuous increasing functions
$$\tilde{g}_n(v):=\int_{X'}\gamma_n(v,\tau)d\mu(\tau)\int_{\R}K(s,\tau)ds, \quad n=1,2,3 \dots$$ 
Since $\gamma_{n+1}(s,\tau)\leq \gamma_n(s,\tau), \ n =1,2,3 \dots$, the sequence $\{\tilde{g}_n\}$ is monotone. 
Now, for each fixed $v\geq 0$, we have that  $
\lim_{n \to +\infty} \tilde{g}_n(v) = \tilde{g}(v)$ where $\tilde{g}$ was defined in $N2$.  Observe that  $\tilde{g}$ is also continuous and therefore, by Dini's monotone convergence theorem,   $\tilde{g}_n$ converges  to $\tilde{g}$ uniformly on compacts.   
\begin{lemma} Let $G'(0)>1 $ be a finite number. Then 
$\Theta_n(v) := v - \tilde{g}_n(v)$ is strictly increasing in $v$.  Furthermore,
$G_n(v): = \Theta_n^{-1}(C\gamma_n(v,\tau_0))$ converges to $G(v)$ uniformly on $[0,\zeta_2]$ and 
$G'_n(0)= G'(0) >1$.  Finally, equation $G_n(c)=c$ does not have solutions on $(0,\zeta_1]$.
\end{lemma} 
\begin{proof} Set $w(\tau): = \int_{\R}K(s,\tau)ds$. Since $G'(0)$ is finite, we have that 
$$
\tilde g'(0) = \int_{X'}g'(0,\tau)w(\tau)d\mu(\tau) <1. 
$$
Now, if $v \in [0,1/n]$ then $\tilde{g}_n(v) =  \tilde g'(0)v$ and therefore 
$
\tilde{g}_n(v_2)- \tilde{g}_n(v_1) = \tilde{g}'(0)(v_2- v_1) < v_2- v_1
$ for $0\leq v_1 < v_2 \leq 1/n$. 

Next, for $1/n \leq v_1 < v_2$ we consider the following measurable subsets of $X'$: 
$$
A_j:= \left\{\tau \in X': g(v_j,\tau) \leq \frac{g'(0,\tau)}{n}\right\}, \quad
B_j:= \left\{\tau \in X': g(v_j,\tau) > \frac{g'(0,\tau)}{n}\right\}. 
$$
Clearly,  $B_j= X' \setminus A_j, \ A_2 \subset A_1,  B_1 \subset B_2$ and $B_2\setminus B_1 = 
A_1\setminus A_2$.  We have 
$$
\tilde{g}_n(v_2)- \tilde{g}_n(v_1) = \int_{B_2\setminus B_1}(g(v_2,\tau) - \frac{g'(0,\tau)}{n})w(\tau)d\mu(\tau) +
$$
$$
\int_{B_1}(g(v_2,\tau)-g(v_1,\tau))w(\tau)d\mu(\tau) \leq \int_{B_2}(g(v_2,\tau)-g(v_1,\tau))w(\tau)d\mu(\tau) \leq 
$$
$$
\int_{X'}(g(v_2,\tau)-g(v_1,\tau))w(\tau)d\mu(\tau) = \tilde g(v_2) - \tilde g(v_1) < v_2 -v_1.  
$$
Finally, consider $v_1 < 1/n<v_2$. Then 
$$
\tilde{g}_n(v_2)- \tilde{g}_n(v_1) = \tilde{g}_n(v_2)- \tilde{g}_n(1/n) + \tilde{g}_n(1/n)- \tilde{g}_n(v_1) <   
v_2-1/n + 1/n -v_1 = v_2-v_1.   
$$
This proves that $\Theta_n$ are strictly increasing. Moreover, since clearly $\Theta_n(\zeta_2) > \max_{s \geq 0} C\gamma_n(v,\tau_0)$ for all large $n$, the functions $G_n$ are well defined. The second conclusion of the lemma follows now immediately from the uniform convergence properties of the sequences $\{\gamma_n(v,\tau_0)\}, \ \{\tilde g_n(v)\}$. Note also that $G_n(v)= G'(0)v$ in some small neighborhood $U_n$ of $v=0$. 
Finally, to prove the last conclusion of the lemma,  we observe that  $G_n(c)=c$ implies  
\begin{equation*}
c=\int_X \gamma_n(c,\tau)w(\tau)d\mu(\tau)  \geq \int_X g(c,\tau)w(\tau)d\mu(\tau) = \tilde g(c) + g(c,\tau_0)w(\tau_0).  
\end{equation*}
In this way, $\Theta(c) \geq  g(c,\tau_0)w(\tau_0)$ so that $c \geq G(c)$. Since $G(s) >s$ on $[0,\zeta_1]$ (see Lemma \ref{ran}.3), we conclude that also $G_n(s) >s$ for $s \in [0,\zeta_1]$. 
\end{proof}
\begin{corollary} \label{ish} For all sufficiently large $n$, and with the same $\zeta_1$ and $\zeta_2$ as in 
Lemma \ref{ran}, each $G_n$ possesses all three properties listed in Lemma \ref{ran}. 
\end{corollary}
Hence, for each large $n$,  Corollary \ref{ish}, Theorems \ref{thex} and \ref{thm:per_unif} guarantee the existence of a positive continuous  function $\varphi_n(t)$ such that $\varphi_n(-\infty)=0$, $\liminf_{t\rightarrow +\infty}\varphi_n(t)\geq \zeta_1$, $\varphi_n(t)\leq \zeta_2$, $t\in\R,$ and
$$\varphi_n(t)=\int_X d\mu(\tau)\int_\R K(s,\tau)\gamma_n(\varphi_n(t-s),\tau)ds.$$
Since the shifted functions $\varphi_n(s+a)$ satisfy the same integral equation, we can assume that $\varphi_n(0)=0.5 \zeta_1$. Furthermore, similarly to (\ref{esc}) we can show that the sequence $\{\varphi_{n}\}$ is 
equicontinuous on $\R$.  Consequently there exists  a subsequence $\{\varphi_{n_j}\}$ which converges uniformly on compacts to some bounded element $\phi\in C(\R,\R)$. By the Lebesgue's dominated convergence theorem,  $\phi$ satisfies  equation (\ref{17}). Finally, notice that $\phi(0)=0.5 \zeta_1$ and thus $\phi(-\infty)=0$ and $\liminf_{t\rightarrow+\infty}\phi(t)> \zeta_1$ (by Theorem \ref{thm:per_unif}). This finalizes the proof of Theorem \ref{thex} when $\chi(z)$ has a positive zero. Its statement for $\chi(z)$ having  a negative zero 
is  immediate after  the change of variables $\psi(t) = \phi(-t)$.  
\hfill $\square$
\section{Applications} \label{apn}
\subsection{Co-existence of expansion and extinction waves in evolution equations with asymmetric non-local response \cite{AGT,fhw,gouss,ma1,tz,TAT}.} \hfill 

Here we  complement   studies \cite{AGT,tz,TAT} concerning  positive bounded wavefronts $u(x,t)= \phi(x+ct)$ for the non-local delayed
reaction-diffusion equation
\begin{equation}\label{17a} \hspace{0mm}
u_t(t,x) = u_{xx}(t,x)  - f(u(t,x)) + \int_{\R}K(x-y)g(u(t-h,y))dy, \ u
\geq 0, 
\end{equation}
where  

({$\mathcal F$}) {\it  locally Lipschitzian function $f: \R_+ \to \R_+, f'(0)>  f(0)=g(0)=0,$ is   strictly increasing and $f(+\infty) > \sup_{s\geq 0} g(s)$. In addition,  $f'(0) < g'(0) < +\infty$ and  $g(t)>0, $ $ t >0$. Kernel $K\geq 0$ is normalized by  $\int_\R K(s)ds=1$.}

We admit  spatial asymmetry  of equation (\ref{17a}) by considering non-even 
kernels.  Due to this circumstance,  the concept
of  wavefront needs some clarification. Indeed, in the symmetric case, the following 
two equivalent definitions have been commonly used: 1)
{\it wavefront $u(x,t) =
\phi(x-ct)$ is a positive classical solution  of (\ref{17a})
satisfying $\phi(-\infty) = \kappa,$ $\phi(+\infty) = 0$}, e.g. see \cite{BNPR,gouS}; 2)   {\it
wavefront $u(x,t) = \psi(x+ct)$ is a positive classical
solution   satisfying $\psi(-\infty) = 0,$
$\psi(+\infty) = \kappa$}, e.g. see \cite{fhw,tz}.   If $K(s)\equiv K(-s)$,  both
definitions  define the same object since   wavefront $\phi(x-ct)$ generates
wavefront $\psi(x+ct):= \phi(-(x+ct))$. Moreover, the propagation speed $c$ should be positive in each
of the above definitions if  $K$ is an even function. Therefore, from the biological point of view the both type of
wavefronts can be interpreted as {\it  the expansion fronts}:
they converge to the positive equilibrium at each fixed position $x$
as $t \to +\infty$. 

Taking into account the above discussion, we will  use  more general definition adapted to
the possible asymmetry $K$:
{\it Bounded positive classical solution $u(x,t) =
\phi(x+ct)$  of equation (\ref{17a}) is a semi-wavefront if either
$\phi(-\infty) = 0$ or $\phi(+\infty) = 0$.}
The prefix {\it semi} means here that, contrary to the wavefronts,  the
convergence of  $\phi(t)$ at the complementary end of $\R$ is not
mandatory.  It is clear that $u(x,t) =
\phi(x+ct)$  is a semi-wavefront if and only if $\phi(t)$ is a positive bounded $C^2-$solution
of the integro-differential equation
\begin{equation} \label{twe}
y''(t)-cy'(t)-f(y(t)) + \int_{\R}K(s)g(y(t-s-ch))ds =0,
\end{equation}
which vanishes either at $-\infty$ or at $+\infty$.  By abusing the notation, we 
still call such a solution $y=\phi(t)$ a semi-wavefront. 
Equation (\ref{twe}) can be written as 
$$
 y''(t) - cy'(t)-\beta y(t)+f_\beta(y(t))+  \int_{\R}k_h(w)g(y(t-
w))dw=0,\, t \in \R,
$$
where  $k_h(w) = K(w-ch)$ and $f_\beta(s)=\beta s- f(s)$ for some  $\beta> 0$.  
Then  the wave profile $\phi$ solves the equation
$$
\phi(t)=\frac{1}{\sigma(c)}\left(
\int_{-\infty}^te^{\nu(t-s)}(\mathcal G \phi)(s)ds
+\int_t^{+\infty}e^{\mu(t-s)}(\mathcal G\phi)(s)ds\right), 
$$
where  $\sigma(c)=\sqrt{c^2+4\beta}$,  $\nu<0<\mu$ are the roots of  $z^2-cz-\beta=0$
 and  $(\mathcal G\phi)(t):=\displaystyle \int_{\R}k_h(s)g(\phi(t-
s))ds+f_\beta(\phi(t))$, e.g. see  \cite{AGT}.  In other words, 
$$ 
\phi(t)=(\mathcal K*k_h)*g(\phi)(t)+\mathcal K* f_\beta(\phi)(t),$$
where $
 \mathcal K(s)= e^{\nu s}/\sigma(c)$ for $s\geq 0$, $
 \mathcal K(s)= e^{\mu s}/\sigma(c)$ for $s\leq 0$, and consequently  $\int_\R  \mathcal K(s)ds =1/\beta$. 
We may invoke now Theorems \ref{thm:per_unif},  \ref{thm:existencia}  where $X=\{\tau_0,\tau_1\}$ and 
$$K(s,\tau)=\left\{\begin{array}{cc} (\mathcal K*k_h)(s),& \tau=\tau_0, \\  \mathcal K(s),& \tau=\tau_1, \end{array}\right.\quad  g(s,\tau)=\left\{\begin{array}{cc}  g(s), & \tau=\tau_0, \\ f_\beta(s),& \tau=\tau_1.\end{array}\right.$$ 
Observe that  the functions $g(u,\tau_j), K(s,\tau_j)$  meet ({\bf N}).  Indeed,   there exists  $\zeta_2$ such that  $f(\zeta_2)>\sup_{s \geq 0} g(s)$. Then we can take $\beta$ large enough to have the function $\tilde{g}(s) = f_\beta(s)/\beta = s - f(s)/\beta$  increasing on $[0, \zeta_2]$.  Next, 
$\Theta(v):=v-\tilde{g}(v) = f(v)/\beta$ is strictly  increasing by ({$\mathcal F$}) and $G(s) = f^{-1}(g(s))$ is well defined. Finally, 
$$
 \chi(z) =  \frac{-\chi_1(z,c)}{\beta+cz-z^2}, \ {\rm where \ }
 \chi_1(z,c) = z^2-cz- f'(0)+g'(0)e^{-zch}\int_ {\R}
K(s)e^{-zs}ds.
$$
Analyzing  the mutual position of real zeros of  $\chi_1(z,c)$ and their dependence on the parameter $c$, we establish in Appendix the existence of
two real extended numbers $c_*^-< c_*^+$ called {\it the critical speeds} such that, for every $c \in
(-\infty,c_*^-]\cup[c_*^+, +\infty)$, equation $\chi_1(\lambda,c) =0$
either (i) has exactly two real roots $\lambda_1(c)\leq \lambda_2(c)$ or
(ii) has exactly one real root $\lambda_1(c)$. Furthermore, each
$\lambda_j(c)$ is positive if $c\geq c_*^+$ and is negative if
$c\leq  c_*^-$. If $c \in (c_*^-, c_*^+)$, then $\chi_1(z,c)> 0$ for all
admissible $z$.  The critical speed $c_*^+$ [ $c_*^-$]  is finite if and only if $\chi_1(\lambda,c)$ is finite for some  $\lambda >0$ [respectively, with some $\lambda <0$].  If the  integral in $\chi_1$  diverges for all $z >0$ [for all $z <0$], we set $c_*^+= +\infty$ [respectively, $c_*^-= -\infty$].
\begin{remark} The above definition of $c_*^\pm$ generalizes the concept of  critical speeds $c_*, c_\# \geq 0$ from \cite{TAT}.  In particular, $c_*=c_*^+, \ c_\# =c_*^-$ if
$c_*^- \geq 0$ and $c_\# =0, \  c_*=\max\{0, c^+_*\}$ if $c_*^-<0$. Thus Theorem \ref{mainAA} below gives a global (i.e. including all $c \in \R$) perspective on the existence/persistence results in \cite{TAT}.    
\end{remark}
Applied to equation (\ref{twe}), Theorem \ref{thm:existencia} yields the following extension of  \cite[Theorem 4.2b]{tz}, \cite[Theorem 1.1]{ma1} and \cite[Theorem 4]{TAT}:
\begin{theorem} \label{mainAA} Assume ({$\mathcal F$}) and   $
g(s) \leq g'(0)s, $ $ f(s) \geq f'(0)s $ for all $s \geq 0.
$
Then equation (\ref{twe}) has at least one semi-wavefront $u= \phi_c(x
+ct)\leq \zeta_2$ for each  $c \in (-\infty, c_*^-]\cup [c_*^+, +\infty)$.
Moreover, if $c \leq c_*^-$ then  $\phi_c(+\infty) =0$ and
$\liminf_{s \to -\infty} \phi_c(s) >\zeta_1$. Similarly, if $c \geq c_*^+$
then  $\phi_c(-\infty) =0$ and $\liminf_{s \to +\infty}
\phi_c(s) > \zeta_1$. Next, if  equation $f(s)=g(s)$ has only two solutions: $0, \kappa$,  with $\kappa$ being globally attracting with respect to the map $f^{-1}\circ g : (0,\zeta_2] \to (0,\zeta_2] $,  then each of these semi-wavefronts  is in fact  a wavefront. 
\end{theorem}
\begin{proof} For a fixed $c' \in \R \setminus [c_*^-, c_*^+]$, this result follows from Theorem \ref{thm:existencia} since the equation $\chi_1(z,c') =0$ has at least one real root in the interior of the domain of definition of $\chi_1(\cdot,c')$.  Now, if $c' \in \{c_*^-, c_*^+\}$ is finite,  we obtain a semi-wavefront $\phi_{c'}$ as a limit of profiles $\phi_{c_j}$ where either  $c_j\uparrow c_*^- $ or  $c_j\downarrow c_*^+$.  See  Section \ref{SE}.2 above or \cite[Section 6, Case II]{TAT} for more details. 
\end{proof} 
We observe that   each possible mutual position of
$c_*^-\leq c_*^+$ and $0$ is possible.  For instance, if 
$K(s)= e^{-(s+\rho)^2}/\sqrt{4\pi}, \quad h=2, \quad g'(0)=2 > f'(0)=1,$
then $c_*^+=-c_*^-= 0.79$ for $\rho=0$ (symmetric case), while
$c_*^+= 2.7, \ c_*^-=  0.7\dots$ for $\rho=5$ (asymmetric case).
In particular, if $\rho =5$ then equation (\ref{twe}) has at least one {\it stationary}
(i.e. propagating at the velocity
$c=0$) semi-wavefront.
In the case when $c_*^-,c_*^+$ are of the same sign, an interesting (by its possible biological interpretation) phenomenon occurs:
equation (\ref{twe}) can possess the {\it extinction} waves. 
Indeed, if $0<c< c_*^-$ then the wave $u(x,t) = \phi(x+ct)$
converges to $0$ at each position $x$ as $t \to +\infty$.
Analogously, for each $x\in \R$, we have $\lim_{t \to -\infty}
u(x,t)=0$ when the velocity $c$ is such that $c_*^+<c<0$.  As far as we know, this kind of  
extinction waves was for the first time mentioned by K. Schumacher as {\it backward traveling fronts} in \cite[p. 66: Example and Figure 3]{KS}.  See also \cite{CDM,FZ,LZ}.

Finally, under weaker conditions on $g, f$, we get  from Theorem \ref{thm:per} the following 
\begin{theorem} \label{41a22}  Assume ({$\mathcal F$})  and let $u= \phi(x+ct)$ be a positive
bounded solution of equation (\ref{twe}) satisfying $
\liminf\limits_{s \to -\infty} \phi(s)
=0$.  Then $\phi(-\infty)=0$,  the critical speed  $c_*^+$ is  finite and $c \geq c_*^+$.
A similar result holds when $
\liminf\limits_{s \to +\infty} \phi(s)
=0$. 
Hence,  equation (\ref{twe}) does not have neither  pulses nor semi-wavefronts propagating at the velocity $c \in (c_*^-,c_*^+) $.
\end{theorem}

\subsection{Nonlocal lattice equations \cite{AGT,fwz,ma1,mazou,TPT,wen,Yu}}\hfill

Let consider semi-wavefronts $w_j(t)=u(j+ct)$ of the nonlocal lattice equation
$$
w'_j(t)= D[w_{j+1}(t)-2w_j(t)+w_{j-1}(t) ]-  dw_j(t) +\sum_{k\in \Z}\beta(j-k)g(w_k(t-r)),  \ j \in \Z, $$
where $\beta(k)\geq 0, \ \sum_{k\in \Z}\beta(k)=1$.
Let  $\pm\gamma^\#_\pm \geq 0$ be extended real numbers such that $\sum_{k\in \Z}\beta(k)e^{-zk}$ converges when $z\in \Gamma^\#:=(\gamma^\#_-,\gamma^\#_+)$ and is divergent when $\pm z > \pm \gamma^\#_{\pm}$.  By Cauchy-Hadamard formula,
$\gamma^\#_+ = - \limsup_{k \to +\infty} k^{-1}\ln\beta(-k)$, where by convention $\ln(0)=-\infty$. A similar formula also holds for $\gamma^\#_-$. 
The wave profile $u$ satisfies
\begin{equation}\label{sy3}
cu'(x)= D[u(x+1)+u(x-1) -2u(x)] - du(x) + \sum_{k\in \Z}\beta(k)g(u(x-k-cr)).
\end{equation}
Let us take now $c\not=0$. Then each positive bounded solution $u$ of (\ref{sy3}) satisfies (\ref{17}) 
with $X=\{\tau_0,\tau_1\}$  and
$$K(s,\tau)=\left\{\begin{array}{cc} D(H_{-1}(s)+H_1(s)), & \tau=\tau_0, \\  \sum_{k\in \Z}\beta(k)H_{k+cr}(s),& \tau=\tau_1, \end{array}\right. \quad  g(s,\tau)=\left\{\begin{array}{cc}  s, & \tau=\tau_0, \\ g(s),& \tau=\tau_1,\end{array}\right.$$
$$H_\tau(t)= |c|^{-1}e^{-\frac{2D+d}{c}(t-\tau)}\chi_{\R_+}(({\rm sign}\,c\,) (t-\tau)), \quad
\chi(z,c): =\tilde\chi(z,c)(2D+d+cz)^{-1}, $$
$$ \tilde\chi(z,c):= d+2D+cz -D(e^z +e^{-z})-g'(0)e^{-crz}\sum_{k\in \Z}\beta(k)e^{-kz}, \  d+2D+cz  >0.$$
The following statement  can be proved analogously to Lemma \ref{L23} in Appendix:  
\begin{lemma} \label{L23*} Assume that $\pm\gamma^\#_\pm  >0$  and that $g'(0) > d$. 
Then there exist 
real numbers $c_*^-<c_*^+$ such that, for every $c \in \frak{C}:=
(-\infty,c_*^-]\cup[c_*^+, +\infty)$, equation $\chi(\lambda,c) =0$
either (i) has exactly two real roots $\lambda_1(c)<\lambda_2(c)$ or
(ii) has exactly one real root $\lambda_1(c)$. Furthermore, each
$\lambda_j(c)$ is positive if $c\geq  c_*^+$ and is negative if
$c\leq  c_*^-$. If $c \in (c_*^-, c_*^+)$, then $\chi(z,c)> 0$ for all
$z \in (\gamma^\#_-, \gamma^\#_+)$.
\end{lemma}
\begin{proof}  See the proof of Lemma \ref{L23} below where it suffices to consider, instead of (\ref{tr}), the equation  
$$
d+2D+cz -g'(0)e^{-crz}\sum_{k\in \Z}\beta(k)e^{-kz} = D(e^z +e^{-z}).
$$
\end{proof}
A formal computation shows that $\tilde g(s)= 2Ds/(2D+d),$ $\theta (s) = ds/(2D+d),$  
$G(s)= g(s)/d$.  Therefore, in complete analogy with the previous subsection,  Theorem \ref{thm:existencia} 
yields the following
\begin{theorem}\label{p14} Let  $G(s) = g(s)/d$ has properties 1-3 listed in  Lemma \ref{ran} and  $
g(s) \leq g'(0)s$ for all $s \geq 0.
$
Then,  for every  $c \in \frak{C}\setminus\{0\}$,  the lattice 
equation has at least one semi-wavefront $u_j(t)= \phi_c(j
+ct)\leq \zeta_2$.  The profile $\phi_c$  shares every  property mentioned 
in the conclusion part of Theorem \ref{mainAA} (with $f=id$). 
\end{theorem}
Theorem \ref{p14} extends  \cite[Theorem 3.1]{wen}, \cite[Theorem 2.1]{mazou}, 
 \cite[Theorem 5.4]{XLZ} and  \cite[ Theorem 4.1]{fwz1} for non-monotone $g$ and 
 asymmetric $\beta$.   
\section{Appendix}
Consider
$
 \psi(z,c) = z^2-cz-q+pe^{-zch}\int_ {\R}
K(s)e^{-zs}ds,
$ where  $p> q$ and 
$K \geq 0,$ $\int_\R K(s)ds = 1$.
\begin{lemma} \label{L23} Assume that $p > q >0$ and that $\psi(z,c)$ is defined for
all $z$ from some maximal open interval $(a,b) \ni 0$.
Then there exist 
real numbers $c_*^-<c_*^+$ such that, for every $c \in
(-\infty,c_*^-]\cup[c_*^+, +\infty)$, equation $\psi(\lambda,c) =0$
either (i) has exactly two real roots $\lambda_1(c)<\lambda_2(c)$ or
(ii) has exactly one real root $\lambda_1(c)$. Furthermore, each
$\lambda_j(c)\in (a,b)$ is positive if $c\geq c_*^+$ and is negative if
$c\leq  c_*^-$. If $c \in (c_*^-, c_*^+)$, then $\psi(z,c)> 0$ for all
$z \in (a,b)$.
\end{lemma}
\begin{proof}  Since $\psi''_z(z,c) > 0, \ z \in (a,b)$, we conclude that
$\psi(z,c)$  is strictly convex with respect to $z$.  Consequently,
the equivalent equation
\begin{equation}\label{tr}
(H(z,c):=) (q+cz-z^2)e^{zch}=
p\int_{\R}e^{-zs}K(s)ds \quad(:= G(z)).
\end{equation}
has at most two real roots. Since $\psi(0,c)=p-q>0$, the convexity
of $\psi$ guarantees that these roots (whenever exist) are of the
same sign. Next, we have that $G(0)=p>0, \ G''(z)
> 0$, $G(z)>0, \ z \in (a,b)$. The left hand side of (\ref{tr}) increases
to $+\infty$ [converges to $0$ ] at each $z\in (0,b)$ when $c$ tends to
$+\infty$  [to $-\infty$ respectively] and the right hand side does
not depend on $c$.  Moreover,  the left hand side of (\ref{tr})
increases with respect to $c$ at every positive point $z$
where $q+cz-z^2 >0$. In consequence,
if equation (\ref{tr})  has a positive root for
some $c= c'$,  then it does  have a positive root  for each $c > c'$.
All this implies the existence of  $c_*^+$ such
that the equation (\ref{tr}) have either two positive roots
$\lambda_1(c)\leq\lambda_2(c)$ or a unique positive root
$\lambda_1(c)$ if and only if $c\geq c_*^+$.
In fact, an easy analysis of (\ref{tr}) shows that the
positive $\lambda_1(c)$ exists and depend
continuously on $c$ from the maximal open interval $(c_*^+,\infty)$.

Similarly, the left hand side of (\ref{tr}) increases to $+\infty$ [converges to $0$] at each $z\in (a,0)$ when $c$ tends to $-\infty$ [to $+\infty$ respectively]. Moreover,  the left hand side of (\ref{tr})
decreases with respect to $c$ at every  $z<0$
where $q+cz-z^2 >0$. This implies the existence of $c_*^-$ such that the equation (\ref{tr}) has either two negative roots  $\lambda_1(c)\leq\lambda_2(c)$ or a unique negative root
$\lambda_2(c)$  if and only if $c\leq c_*^-$. Again the negative $\lambda_2(c)$ exists and depend
continuously on $c\in(-\infty,c_*^-)$.

The above considerations also shows that $c_*^-$ and $c_*^+$ are finite, and $c_*^- < c_*^+$.
\end{proof}
\begin{remark} \label{RL23} With the unique exception ($c_*^-=-\infty$), all
conclusions of  Lemma \ref{L23} hold also true   in the case when
$(a,b) =(0,b), \ b >0$.  To prove the finiteness of $c_*^+$, it
suffices to observe that for every positive $\delta$ there exists
$c_1 <0$ such that $H(z,c) <0$ for all $z > \delta, \ c < c_1$ and
$H(z,c) <p$ for all $z \in (0,\delta), \ c < c_1$. A similar
assertion (with $c_*^+=+\infty$ ) is valid when $(a,b)= (a,0), \  a
<0$.
\end{remark}
\section*{Acknowledgments}
This research was supported by FONDECYT 1070127 (H. P.)  and 
1110309 (C. G. and S. T.).  H. Prado  and S. Trofimchuk  acknowledge support from 
DICYT USACH 041133PC  and  CONICYT  through PBCT program ACT-56, respectively.  
C. Gomez was also  supported by CONICYT programs    
"Pasant\'ias doctorales en el extranjero" and "Becas para Estudios de Doctorado en Chile".

\medskip
\medskip


\begin{thebibliography}{30}
\bibitem{AGT}    
\newblock  { M.~Aguerrea, C.~ Gomez and S.~Trofimchuk}, 
   \newblock \emph{ On uniqueness of semi-wavefronts (Diekmann-Kaper theory of a nonlinear convolution equation re-visited),} 
      \newblock   Math. Ann., DOI 10.1007/s00208-011-0722-8.  



\bibitem{AR}  \newblock  C. Atkinson and G. E. H. Reuter, 
\newblock  \emph{ Deterministic epidemic waves,} 
\newblock  Math. Proc. Cambridge  Philos. Soc.,  \textbf{80} (1976), 315--330. 


\bibitem{BNPR} 
\newblock  H. Berestycki, G. Nadin, B. Perthame and L. Ryzhik,  
\newblock \emph{ The non-local Fisher-KPP equation: travelling 
waves and steady states},  
\newblock Nonlinearity, \textbf{22} (2009),  2813-2844 .


\bibitem{CDM} 
\newblock  J. Coville, J. D\'avila  and S. Mart\'inez, 
\newblock  \emph{Nonlocal anisotropic dispersal with monostable nonlinearity}, 
\newblock   J. Diff. Eqns.,  \textbf{244} (2008),  3080--3118. 

\bibitem{dkMB} 
\newblock O. Diekmann, 
\newblock \emph{Thresholds and travelling waves for the geographical spread of infection,} 
\newblock  J. Math. Biol., \textbf{6} (1978), 109-130.



\bibitem{dk} 
\newblock O. Diekmann and H. G. Kaper, 
\newblock \emph{On the bounded solutions of a nonlinear
convolution equation,} 
\newblock Nonlinear Anal., \textbf{2} (1978), 721-737.



\bibitem{FZ}  
\newblock {J. Fang and X.-Q. Zhao,}, 
\newblock \emph{Bistable traveling waves for monotone semiflows with applications,} 
\newblock preprint 	\arXiv{1102.4556v1}.


\bibitem{fwz}
\newblock {J. Fang,  J. Wei and  
X.-Q. Zhao, } 
\newblock  \emph{
 Uniqueness of traveling waves for nonlocal lattice equations,}
 \newblock  {Proc. Amer. Math. Soc. }
{\bf 139}  (2011), 1361--1373.

\bibitem{fwz1}  \newblock J. Fang, J. Wei and X.-Q. Zhao, 
\newblock  \emph{ Spreading speeds and traveling waves for non-monotone time-delayed lattice equations, }\newblock  {Proc. R. Soc.  A,} \textbf{466} (2010) 1919--1934.

\bibitem{fhw} \newblock  T. Faria, W. Huang and J. Wu, \newblock  \emph{ Traveling
waves for delayed  reaction-diffusion equations with non-local
response,} {Proc. R. Soc.  A },  \textbf{ 462} (2006) 229--261.


\bibitem{GBT}  \newblock  K. Gopalsamy, N. Bantsur and S.
Trofimchuk, \newblock  \emph{A note on global attractivity in models of
hematopoiesis,} 
\newblock  Ukrainian Math. J., \textbf{50} (1998)  5--12.

\bibitem{gouS} \newblock   { S.A. Gourley and J.
W.-H. So,} \newblock  \emph{Extinction and wavefront propagation in a
reaction-diffusion model of a structured population with distributed
maturation delay,} 
\newblock  {Proc. Royal Soc. of Edinburgh}, \textbf {133A} (2003)
527--548.



\bibitem{gouss} \newblock  {S. A. Gourley, J. So and J. Wu,} 
\newblock  \emph{ Non-locality of reaction-diffusion
equations induced by delay: biological modeling and nonlinear
dynamics,} 
\newblock  {J. Math. Sciences}, \textbf{124} (2004) 5119--5153.


\bibitem{XLZ} 
\newblock  {X. Liang and X.-Q. Zhao,} 
\newblock  \emph{Asymptotic speeds of spread and traveling waves 
for monotone semiflows with applications}, 
\newblock  {Comm. Pure Appl. Math., }, \textbf{60} (2007),
1--40. 

\bibitem{ma1} 
\newblock  { S. Ma,} 
\newblock  \emph{Traveling waves for non-local delayed diffusion equations
via auxiliary equations}, 
\newblock  {J. Diff. Eqns.}, \textbf{237} (2007),
259--277.



\bibitem{mazou} \newblock{ S. Ma and  X. Zou,} 
\newblock \emph{ Existence, uniqueness and stability of travelling
waves in a discrete reaction-diffusion monostable
equation with delay,} 
\newblock {J. Diff. Eqns.}, \textbf{217} (2005),
54--87.



\bibitem{KS}
\newblock  K. Schumacher, 
\newblock \emph{ Travelling-front solutions for integro-differential equations. I }, 
\newblock  J. Reine Angew. Math.,  {\bf 316} (1980),  54--70.  


\bibitem{tz} \newblock {H.R. Thieme and X.-Q. Zhao,}  Asymptotic speeds of spread and
traveling waves for integral equations and delayed
reaction-diffusion models, \newblock  \emph{ J. Diff. Eqns.}, \textbf{195} (2003)
430--470.


\bibitem{TAT}  \newblock { E. Trofimchuk, P. Alvarado and  S. Trofimchuk},
\newblock \emph{
On the geometry of wave solutions of a delayed reaction-diffusion
equation}, 
\newblock {J. Diff. Eqns.},  \textbf{246}  (2009)  1422--1444.

\bibitem{TPT}  
\newblock { E. Trofimchuk, M. Pinto and  S. Trofimchuk}, 
\newblock \emph{Pushed traveling fronts in monostable  equations with  monotone delayed reaction,} 
\newblock preprint 	\arXiv{1111.5161v1}.


\bibitem{VL} \newblock D. Volkov and R. Lui, 
\newblock \emph{Spreading speed and traveling wave solutions of a partially sedentary population}, 
\newblock  IMA J. Appl. Math., \textbf{72} (2007)  801-816.


\bibitem{WBR} \newblock H. Weinberger, 
\newblock \emph {Long time behavior of a class of biological models,} 
\newblock SIAM J. Math. Anal., \textbf{13} (1982) 353--396.



\bibitem{wen} 
\newblock P. Weng, H. Huang and  J. Wu, 
\newblock \emph{Asymptotic speed of propagation of wave fronts in a lattice delay differential equation with global interaction,} 
\newblock IMA J. Appl. Math., \textbf{ 68} (2003) 409--439.
 

\bibitem{Yu} 
\newblock Z.-X. Yu,
\newblock \emph{Uniqueness of critical traveling waves for nonlocal lattice equations with delays,} 
\newblock Proc. Amer. Math. Soc.,  (2012),  published on-line. 


\bibitem{LZ} 
\newblock L. Zhang and B. Li, 
\newblock \emph{Traveling wave solutions in an 
integro-differential competition model, } 
\newblock Discrete Contin. Dyn. Syst. - B, \textbf{ 17} (2012) 417--428.



\end{thebibliography}
\end{document}